\documentclass[reqno]{amsart}
\usepackage{amsfonts,amsmath,amssymb,enumerate}

\newtheorem{theorem}{Theorem}[section]
\newtheorem{lemma}[theorem]{Lemma}
\newtheorem{corollary}[theorem]{Corollary}

\theoremstyle{definition}

\newtheorem{remark}[theorem]{Remark}

\numberwithin{equation}{section}
\newcommand{\etext}[1]{\quad\emph{#1}\quad}
\newcommand{\itext}[1]{\quad\text{#1}\quad}

\newcommand{\8}{\infty}

\newcommand{\bL}{\mathbb L}
\newcommand{\bR}{\mathbb R}
\newcommand{\bS}{\mathbb S}
\newcommand{\bZ}{\mathbb Z}
\newcommand{\con}{{\rm const}}

\newcommand{\eq}{\equiv}
\newcommand{\er}{\eqref}

\newcommand{\il}{\int\limits}
\newcommand{\iil}{\iint\limits}
\newcommand{\ol}{\overline}

\newcommand{\pa}{\partial}

\newcommand{\rn}{{\mathbb R}^n}

\newcommand{\sign}{{\rm sign\,}}
\newcommand{\su}{\subset}

\newcommand{\tl}{\tilde}
\newcommand{\ve}{\varepsilon}

\DeclareMathOperator*{\osc}{osc}

\author{Gong Chen}
\address{Department of Mathematics, University of Chicago,
Chicago, Illinois 60637, USA}
\email{gc(at)uchicago.edu}

\author{Mikhail Safonov}
\address{School of Mathematics, University of Minnesota,
Minneapolis, Minnesota 55455, USA}
\email{safonov(at)math.umn.edu}

\begin{document}

\title{On Second Order Elliptic and Parabolic Equations of Mixed Type}


\begin{abstract}
It is known that solutions to second order uniformly elliptic and parabolic equations, either in divergence or nondivergence (general) form, are H\"{o}lder continuous and satisfy the interior Harnack inequality. We show that even in the one-dimensional case ($x\in\bR^1$), these properties are not preserved for equations of mixed divergence-nondivergence structure: for elliptic equations.
\begin{equation*}
    D_i(a^1_{ij}D_ju)+a^2_{ij}D_{ij}u=0,
\end{equation*}
and parabolic equations
\begin{equation*}
    p\pa_t u=D_i(a_{ij}D_ju),
\end{equation*}
where $p=p(t,x)$ is a bounded strictly positive function. The H\"{o}lder continuity and Harnack inequality are known if $p$ does not depend either on $t$ or on $x$. We essentially use homogenization techniques in our construction. Bibliography: 23 titles.
\end{abstract}

\maketitle

\section{Introduction}\label{S1}

Leaving aside the elliptic and parabolic equations with ``regular'' coefficients, and also the cases of lower dimension, the H\"{o}lder regularity of solutions was first proved in 1957 by De Giorgi \cite{DG} for uniformly elliptic equations, and soon afterwards by Nash \cite{N} for more general uniformly parabolic equations in the {\em divergence form}
\begin{equation}\label{D}\tag{D}
    Lu:=-\pa_t u+D_i\big(a_{ij}D_ju\big)=0,
\end{equation}
where $\pa_t u:=\pa u/\pa t,\; D_i:=\pa/\pa x_i\,$ for $i=1,2,\ldots, n$, and the equation is understood in the integral sense, i.e. $u$ is a {\em weak solution} of \er{D}. Here and throughout the paper, we assume the summation convention over repeated indices $i,j$ from $1$ to $n$. The coefficients $a_{ij}=a_{ij}(t,x)$ are real Borel measurable functions satisfying the {\em uniform parabolicity condition}
\begin{equation}\label{U}\tag{U}
    a_{ij}\xi_i\xi_j\ge \nu\,|\xi|^2
    \itext{for all}\xi\in\bR^n,
    \itext{and}\sum_{i,j}a_{ij}^2\le\nu^{-2},
\end{equation}
with $\nu=\con\in(0,1]$. The elliptic equations $D_i\big(a_{ij}D_ju\big)=0$ can be formally considered as a particular case of \er{D}, in which $a_{ij}$ and $u$ do not depend on $t$, so that $\pa_t u=0$.
\smallskip

The interior Harnack inequality was first proved in 1961 by Moser \cite{M1} for elliptic equation, and then in 1964 \cite{M2}
for more general parabolic equations \er{D} (see Theorem \ref{T1.1} below).
\smallskip

The Harnack inequality together with H\"{o}lder regularity of solutions to uniformly parabolic equations in the {\em nondivergence form}
\begin{equation}\label{ND}\tag{ND}
    Lu:=-\pa_t u+a_{ij}D_{ij}u=0
\end{equation}
was proved much later, at the end of 1970th, by Krylov and Safonov \cite{KS}. See \cite{K, GT, Li, NU} and references therein for further history, generalizations to equations with unbounded lower order terms, and various applications. The method in \cite{KS} is based on some variants of {\em growth lemmas}, which were originally introduced by Landis \cite{La}.
\smallskip

In its simplest formulation (without involving lower order terms), these results for parabolic case can be summarized as follows. In the elliptic case, one can simply drop the dependence on $t$: $a_{ij}=a_{ij}(x)$ and $u=u(x)$.

\begin{theorem}[Harnack inequality]\label{T1.1}
    For $y\in\bR^n, Y:=(s,y)\in \bR^{n+1}$, and $r>0$, denote
    \begin{equation}\label{1.1}
    B_r(y):=\{x\in\bR^n:\; |x-y|<r\},\qquad
    C_r(Y):=(s-r^2,s)\times B_r(y).
    \end{equation}
    Let $u$ be a non-negative solution of \er{D} or \er{ND} in $C_{2r}(Y)$. Then
    \begin{equation}\label{1.2}
        \sup_{C_r(Y_r)}u\le N\cdot \inf_{C_r(Y)}u,
        \etext{where} Y_r:=(s-2r^2,y),
    \end{equation}
    and the constant $N>1$ depends only on $n$ and $\nu$ in \er{U}.
\end{theorem}

\begin{theorem}[Oscillation estimate]\label{T1.2}
    Let $u$ be a bounded solution (not necessarily non-negative) of \er{D} or \er{ND} in $C_{2r}(Y)$ for some $Y\in \bR^{n+1}$ and $r>0$. Then
    \begin{equation}\label{1.3}
    \osc_{C_r}u:=\sup_{C_r}u-\inf_{C_r}u
    \le \theta \cdot \osc_{C_{2r}}u
    \etext{with a constant}
    \theta=\theta(n,\nu)\in (0,1),
    \end{equation}
    where $C_r:=C_r(Y)$.
\end{theorem}

The proof of these theorems, separately in the cases \er{D} and \er{ND}, is contained in the above mentioned references.

\begin{corollary}[H\"{o}lder estimate]\label{C1.3}
    Let $u$ be a bounded solution of \er{D} or \er{ND} in a cylinder $C_r:=C_r(Y)$ for some $Y\in \bR^{n+1}$ and $r>0$. Then
    \begin{equation}\label{1.4}
    \omega_{\rho}:=\osc_{C_{\rho}}u
    \le\Big(\frac{2\rho}{r}\Big)^{\alpha}\omega_r
    \etext{for} 0<\rho\le r,
    \end{equation}
    where $\alpha =\alpha(n,\nu):=-\log_2\theta>0,\quad\theta \in(0,1)$ is the constant in \er{1.3}.
\end{corollary}

\begin{proof}
    Using the estimate \er{1.3} with $\rho \in (0,r/2]$ in place of $r$, we get
    \[ \omega_{\rho}\le \theta \omega_{2\rho}
    \itext{for} 0<2\rho\le r.\]
    For each $\rho\in(0,r]$, there is a unique integer $k\ge 0$ such that $2^{-k-1}r<\rho\le 2^{-k}r$. Iterating the previous inequality and using monotonicity of $\omega_{\rho}$, we get
    \[ \omega_{\rho} \le \theta\omega_{2\rho} \le\cdots\le
    \theta^k\omega_{2^k\rho}
    \le \theta^k\omega_r.\]
    Finally, taking
    $\alpha:=-\log_2\theta>0$, we obtain
    \[ \theta^k=\big(2^{-k}\big)^{\alpha}
    \le\Big(\frac{2\rho}{r}\Big)^{\alpha},\]
    which yields the desired estimate \er{1.4}.
\end{proof}

More recently, Ferretti and Safonov \cite{FS} tried to develop some ``unifying'' techniques which would equally applicable to equations in both \er{D} and \er{ND} forms. They found out the the growth lemmas can serve as a common ground for the proof of the Harnack inequality and other related facts, though the methods of their proof are completely different in these two cases. A natural question arises, whether or not Theorems \ref{T1.1} and \ref{T1.2} hold true, with constants independent on the smoothness of the coefficients, for solutions of mixed type elliptic equations of ``mixed'' type
\begin{equation}\label{1.5}
    D_i(a^1_{ij}D_ju)+a^2_{ij}D_{ij}u=0,
\end{equation}
where matrix functions $a^1_{ij}$ and $a^2_{ij}$ satisfy \er{U}. It was shown \cite[Example 1.7]{FS} by direst calculation that this is not true even in the one-dimensional case, when \er{1.5} is reduced to an ordinary differential equation with highly oscillating coefficients.
\smallskip

In Section \ref{S2} we discuss this phenomenon from the point of view of homogenization theory. Namely, we use the fact that solutions of differential equations with periodic coefficients can be approximated by solutions of equations with constant coefficients as the period tends to $0$. For this purpose, we need to take into consideration periodic solutions to the adjoint equation $L^*v=0$. As a variant of the Fredholm alternative, it is known
(see \cite[Part II, Ch. 3]{BJS}, that in the class of smooth periodic functions, the elliptic equation (possibly of higher order) $Lu=f$ is solvable if and only if $f$ is orthogonal in $\bL^2$ to any nontrivial solution to  $L^*v=0$, i.e. $(f,v)=0$, and moreover, the homogeneous equations $Lu=0$ and $L^*v=0$ have same number of linearly independent solutions. For elliptic equations of second order,
\begin{equation}\label{1.6}
    Lu:=a_{ij}D_{ij}u+b_iD_iu=0,
\end{equation}
from the strong maximum principle it follows that any periodic solution must be constant. Moreover, periodic solutions to  $L^*v=0$ cannot change sign, because otherwise we could choose a positive periodic function $f$ such that $(f,v)=0$. Then we must have a periodic solution to the equation $Lu=f>0$, which is impossible by the strong maximum principle.
\smallskip

In Section \ref{S3}, we extend the Fredholm alternative to second order parabolic equations in the class of functions which are periodic both in $t$ and $x$. Finally, in Section \ref{S4}, we show that even for $x\in\bR^1$, the solutions of the ``mixed'' parabolic equations
\begin{equation}\label{1.7}
    Lu:=-p\pa_t u+D_i(a_{ij}D_ju)=0
\end{equation}
with $p=p(t,x)\in [\nu,\nu^{-1}]$ do not satisfy the Harnack inequality \er{1.2} and H\"{o}lder estimate \er{1.4} with constants independent on the smoothness of coefficients. Note that if $p=p(t)$ does not depend on $x$, then one can simply divide both sides of \er{1.7} by $p$ and replace $a_{ij}$ by $a_{ij}/p$, so that the equation \er{1.7} is reduced to the standard form \er{D}. It is also known from \cite{PE}, see also \cite{FS}, that the Harnack inequality and other related facts are true if $p=p(x)$ does not depend on $t$. One of the key observations here is that the function
    \[ I(t):=\il_{\bR^n}p(x)\,u(t,x)\,dx\]
under reasonable assumptions (which allow differentiation and application of the divergence theorem) satisfies
    \[ I'(t)=\il_{\bR^n}p\,\pa_tu\,dx
    =\il_{\bR^n}D_i(a_{ij}D_ju)\,dx=0,\]
so that $I(t)\equiv I(0)$ for $t\ge 0$. This argument does not work for $p=p(t,x)$.
\smallskip

We essentially use homogenization technique which was introduced by M. I. Freidlin \cite{F} in 1964 form the probabilistic point of view. In the  analytic setting, our method can be considered as an application of the {\em method of asymptotic expansions} (see \cite[Sec. 1.2]{BLP} or \cite[Sec. 1.4]{JKO}, and references therein). More recently, a similar technique was used in a more difficult (non-periodic) case by
N. Nadirashvili \cite{Nd} (see also \cite{S}) for the proof of non-uniqueness of weak solutions to second order elliptic equations $\,a_{ij}D_{ij}u=0\,$ with measurable coefficients $\,a_{ij}$.
\smallskip

In the parabolic case, our approach to the existence of time-periodic solutions is based on a Fredholm type argument, which was also discussed by G. Lieberman \cite{Li1} for more general equations and domains. For our purposes, it is technically simpler that an alternative approach based on the Krein-Rutman theorem (see, e.g. \cite{H}).

We assume that all the functions $a_{ij}, b_i, u, v$, etc., are smooth enough, so that all the derivatives in our formulations and proofs are understood in the classical sense. Our goal is to show that for solutions of the equations \er{1.5} and \er{1.7}, the Harnack inequality, and in fact even the estimates for the modulus of continuity, are in general impossible with constants independent on the smoothness of coefficients.
\smallskip

BASIC NOTATIONS:
\begin{enumerate}[]\itemsep3pt
{\setlength\itemindent{8pt}
\item $x=(x_1,x_2,\ldots,x_n)$ are vectors or points in $\rn$,
    $\;|x|$  is  the length of $x\in\rn$.
\item The balls $B_r(y)$ and cylinders $C_r(Y)$ are defined in \er{1.1}.
\item In Sections \ref{S2} and \ref{S4}, $\bS$ stands for the set of all smooth $1$-periodic functions.
\item The notation $A:=B$, or $B=:A$ means ``$A=B$ by definition''. Throughout the paper, $N$ (with indices or without) denotes different constants depending only on the prescribed quantities such that $n,\nu,$ etc.
    This dependence is indicated in the parentheses: $N=N(n,\nu,\ldots)$.
}
\end{enumerate}

\section{Elliptic equations of mixed type}\label{S2}

In this section, $\bS$ denotes the set of all smooth $1$-periodic functions on $\bR^1$. In the one-dimensional case, the Fredholm alternative for elliptic operators $L$ in \er{1.6} is contained in Theorem \ref{T2.2} below. Part (II) of this theorem can also be considered as a particular case of Theorem \ref{T3.1} in the next section, when there is no dependence on $t$. We still give a direct elementary proof, because some of its details, such as Corollary \ref{C2.4}, are used below. We precede it with a simple lemma.

\begin{lemma}\label{L2.1}

(I) For arbitrary continuous functions $\,a=a(x)>0\,$ and $\,b=b(x)\,$ on $\bR^1$,
the function
\begin{equation}\label{2.1}
    u_0(x):=\il_0^x e^{-F(y)}dy,\etext{where}
    F(y):=\il_0^y \frac{b(z)}{a(z)}\,dz
\end{equation}
satisfies $\,Lu_0:=au_0''+bu_0'=0;\quad
u_0(0)=0,\quad u_0'(0)=1$.
\smallskip

(II) In addition to the functions $a$ and $b$ in the previous part  (I), let a continuous function $f=f(x)$ on $\bR^1$ be given. Then the function
\begin{equation}\label{2.2}
    U(x):=\il_0^x e^{-F(y)}\il_0^y \frac{e^{F(z)}f(z)}{a(z)}\,dz\,dy,
\end{equation}
satisfies $\,LU:=aU''+bU'=f;\quad
U(0)=U'(0)=0$.
\end{lemma}
\smallskip

Note that the functions $u_0$ and $U$ in this lemma are not $1$-periodic in general.
\medskip

\begin{theorem}[Fredholm alternative]\label{T2.2}
Consider the linear differential equation
\begin{equation}\label{2.3}
    Lu:=au''+bu'=f
\end{equation}
with $a,b,f\in\bS$, where $\,a=a(x)>0$, and its formally adjoint homogeneous equation
\begin{equation}\label{2.4}
    L^*v=(av)''-(bv)'=0.
\end{equation}
We have the following properties.
\smallskip

(I) The equation \er{2.4} has a nontrivial solution $v\in\bS$ which is unique up to a multiplicative constant. Moreover, $v(x_0)=0$ at some point $x_0$ if and only if $v\eq 0$. The solution is uniquely defined by the normalization condition
\begin{equation}\label{2.5}
    \il_0^1v(x)\,dx=1.
\end{equation}
\smallskip

(II) The equation \er{2.3} has a solution in $\bS$ if and only if
\begin{equation}\label{2.6}
    f^0:=(f,v):=\il_0^1fv\,dx=0,
\end{equation}
where $v\in\bS$ satisfies \er{2.4}--\er{2.5}. The solution is unique up to an additive constant.
\end{theorem}

\begin{proof}
    (I) Obviously, the equation \er{2.4} is equivalent to
    \begin{equation}\label{2.7}
    (av)'+bv=c=\con.
    \end{equation}
For a given initial value $v(0)$ the unique solution to this equation on $\bR^1$ is given by the expression
\begin{equation}\label{2.8}
    a(x)v(x)=e^{-F(x)}
    \bigg[a(0)v(0)+c\il_0^xe^{F(y)}dy\bigg],
\end{equation}
where the function $F$ is defined in \er{2.1}. There is a unique choice of the constant $\,c\,$ which guarantees the equality $v(1)=v(0)$. In this case, since $a,b,f$ are $1$-periodic, both functions $v(x)$ and $v(x+1)$ satisfy \er{2.7} with the same initial condition at the point $x=0$. By uniqueness for the Cauchy (initial value) problem, we have
\begin{equation}\label{2.9}
    v(x+1)\eq v(x)\itext{on}\bR^1,
    \itext{i.e.} v\in\bS.
\end{equation}
Moreover, from $\,v(0)=v(1)=0$ it follows $c=0$ and $v\eq 0$.
By periodicity, same is true with any point $x_0\in\bR^1$ in place of $0$: from $\,v(x_0)=v(x_0+1)=0$ it follows  $v\eq 0$.
In particular, any non-trivial solution $v$ cannot change sign.
Finally, multiplying $v$ by an appropriate constant, we get a solution $v\in\bS$ satisfying \er{2.5}.
\smallskip

(II) For an arbitrary smooth function $u$, integrating by parts, or equivalently, using the identity
\begin{equation}\label{2.10}
    w'=vLu-uL^*v=vLu,\itext{where}
    w:=u'av-(av)'u+buv,
\end{equation}
we obtain
\begin{equation}\label{2.11}
    (Lu,v):=\il_0^1vLu\,dx=w(1)-w(0).
\end{equation}
Therefore, if $u\in\bS$ satisfies $Lu=f$, then by periodicity, $w(1)=w(0)$, hence $(f,v)=(Lu,v)=0$.
\smallskip

Now suppose that $f\in\bS$ satisfies $(f,v)=0$. We take
$u:=U+\lambda u_0$, where $u_0$ and $U$ are defined in \er{2.1}--\er{2.2}, and a constant $\lambda$ is chosen in such a way that $u(1)=0$. Then
    \[Lu=f\itext{and} u(0)=u(1)=0.\]
From $(f,v)=0$ and \er{2.11} it follows $w(0)=w(1)$, which in turn implies $u'(0)=u'(1)$. Therefore, the functions $u(x)$ and $u(x+1)$ satisfy the equation $Lu=f$ of second order with same initial conditions for $u$ and $u'$ at the point $x=0$. Similarly to \er{2.9},
\begin{equation}\label{2.12}
    u(x+1)\eq u(x)\itext{on}\bR^1,
    \itext{i.e.} u\in\bS.
\end{equation}

The existence of a solution $u\in\bS$ of $Lu=f$ is proved. The uniqueness up to an additive constant follows from the fact that the difference $\tl{u}:=u_1-u_2$ of two solutions of $Lu=f$ satisfies the homogeneous equation $L\tl{u}=0$. Since $\tl{u}$ is also periodic, by the maximum principle we must have $\tl{u}=u_1-u_2=\con$. Theorem is proved.
\end{proof}

\begin{corollary}\label{C2.3}
    For an arbitrary $f\in\bS$, there exists a unique, up to an additive constant, solution in $\bS$ of the equation
    \begin{equation}\label{2.13}
    Lu=f-f^0,\etext{where}
    f^0:=(f,v),
    \end{equation}
    and $v\in\bS$ is defined by \er{2.4}--\er{2.5}.
\end{corollary}

\begin{corollary}\label{C2.4}
    In a special case $f=b$,
    \begin{equation}\label{2.14}
    b^0:=(b,v):=\il_0^1bv\,dx>0
    \etext{if and only if}
    F(1):=\il_0^1\frac{b(x)}{a(x)}\,dx>0.
    \end{equation}
\end{corollary}

\begin{proof}
    Since $a,b,v\in\bS$, from \er{2.7} it follows $\,b^0:=(b,v)=c$. In order to guarantee the equality $\,v(1)=v(0)\,$ in \er{2.8}, we must have $\sign \,c=\sign F(1)$.
\end{proof}

\begin{theorem}\label{T2.5}
    Let $a,b\in\bS$, $a=a(x)>0$, and
    \begin{equation}\label{2.15}
    F(1):=\il_0^1\frac{b(x)}{a(x)}\,dx>0.
    \end{equation}
    For $\ve>0$, denote $a^{\ve}(x):=a(\ve^{-1}x),\;
    b^{\ve}(x):=b(\ve^{-1}x)$, and let $u_{\ve}(x)$ be a solution to the Dirichlet problem
    \begin{equation}\label{2.16}
    L^{\ve}u_{\ve}:=a^{\ve}u''_{\ve}+\ve^{-1}b^{\ve}u'_{\ve}=0;\qquad
    u_{\ve}(0)=0,\quad u_{\ve}(1)=1.
    \end{equation}
    Then $u_{\ve}\to 1$ as $\,\ve\to 0^+$ uniformly on every interval $[\delta,1],\;0<\delta<1$.
\end{theorem}

\begin{proof}
    Denote $\;a^0:=(a,v),\quad b^0:=(b,v)$,
    where $v$ is defined by \er{2.4}--\er{2.5}. We know from \er{2.14} that $b^0>0$. By Corollary \ref{C2.3}, there exist functions $A,B\in\bS$ satisfying
    \[
    LA=a-a^0,\qquad LB=b-b^0.\]
    Denote
    $A^{\ve}(x):=A(\ve^{-1}x),\;B^{\ve}(x):=B(\ve^{-1}x) $.
    Then
    \begin{equation}\label{2.17}
    L^{\ve}A^{\ve}=\ve^{-2}(a^{\ve}-a^0),\qquad
    L^{\ve}B^{\ve}=\ve^{-2}(b^{\ve}-b^0).
    \end{equation}
    For fixed $K=\con>1$ and small $\ve>0$, consider the functions
    \begin{equation}\label{2.18}
    h_K(x):=\frac{1-e^{-Kx}}{1-e^{-K}},
    \itext{and}
    w_{K,\ve}:=u_{\ve}-h_K+g_{K,\ve},
    \end{equation}
    where
    \[ \,g_{K,\ve}:= \ve^2\,\big(A^{\ve}h''_K+\ve^{-1}B^{\ve}h'_K\big).\]
    Using \er{2.16}, \er{2.17}, and the elementary equality
    \[ L^{\ve}(f_1f_2)
    =f_2L^{\ve}f_1+f_1L^{\ve}f_2+2a^{\ve}f'_1f'_2,\]
    we can write
    \[ L^{\ve}w_{K,\ve}=-L^{\ve}h_K+I_1+I_2+I_3,\]
    where
    \begin{eqnarray*}
      I_1 &:= & (a^{\ve}-a^0)h''_K+\ve^{-1}(b^{\ve}-b^0)h'_K
      =L^{\ve}h_K-a^0h''_K-\ve^{-1}b_0h'_K,\\
      I_2 &:= & \ve^{2}\big(A^{\ve}L^{\ve}h''_K
      +\ve^{-1}B^{\ve}L^{\ve}h'_K\big),\\
      I_3 &:= & 2\ve^2a^{\ve}\big((A^{\ve})'h'''_K
      +\ve^{-1}(B^{\ve})'h''_K\big).
    \end{eqnarray*}
    Since $A$ and $B$ are smooth, the derivatives $(A^{\ve})'$ and $(B^{\ve})'$ are of order $\ve^{-1}$.  Therefore, for fixed $K>1$, after cancelation of terms $\pm L^{\ve}h_K$ in the expression for $L^{\ve}w_{K,\ve}$, the remaining terms are bounded by constants independent of $\ve$, except for a negative term $-\ve^{-1}b_0h'_K$ of order $\ve^{-1}$. This term guarantees that for small $\ve>0$, we have $L^{\ve}w_{K,\ve}<0\,$ on $[0,1]$.
    \smallskip

    For arbitrary $\delta$ and $\delta_0$ in $(0,1)$, one can choose $K>1$ such that $1\ge h_K\ge 1-\delta_0$ on $[\delta,1]$, and then choose a small $\ve>0$ such that
    \[ L^{\ve}w_{K,\ve}<0\itext{and}
    |g_{K,\ve}|\le\delta_0\itext{on} [0,1].\]
    We have
    $\,w_{K,\ve}=g_{K,\ve}\ge -\delta_0$ at the ends of the interval $[0,1]$. By the maximum principle, $w_{K,\ve}\ge -\delta_0$ on $[0,1]$, and
    \[ 1\ge u_{\ve}=h_K+w_{K,\ve}-g_{K,\ve}
    \ge 1-3\delta_0
    \itext{on} [\delta,1].
    \]
    Since $\delta_0>0$ can be made arbitrarily small, the theorem is proved.
\end{proof}

As an easy application of this theorem, consider the following generalization of Example 1.7 in \cite{FS}.

\begin{corollary}\label{C2.6}
    Let $a_1$ and $a_2$ be even positive functions in $\bS$ such that
    \begin{equation}\label{2.19}
    \il_0^1\frac{a'_1(x)}{a_1(x)+a_2(x)}\,dx>0.
    \end{equation}
    For $\ve>0$ and $j=1,2$, denote $a_j^{\ve}(x):=a_j(\ve^{-1}x)$, and let $u_{\ve}(x)$ be solutions to the Dirichlet problems
    \begin{equation}\label{2.20}
    L^{\ve}u_{\ve}:=(a_1^{\ve}u'_{\ve})'+a_2^{\ve}u''_{\ve}=0
    \etext{in} (-1,1);\quad
    u_{\ve}(-1)=-1,\quad u_{\ve}(1)=1.
    \end{equation}
    Then
    \begin{equation}\label{2.21}
    u_{\ve}\to \sign x:=
    \begin{cases}
    1 & \etext{if} x>0, \\
    0 & \etext{if} x=0, \\
    -1 & \etext{if} x<0, \\
    \end{cases}
    \etext{as} \ve\to 0^+,
    \end{equation}
    uniformly on every set
    $[-1,-\delta]\cup [\delta,1],\;0<\delta<1$.
\end{corollary}

\begin{proof}
    It is easy to see that the equation $L^{\ve}u_{\ve}=0$ in \er{2.20} can be written in the form \er{2.16} with $a(x):=a_1(x)+a_2(x)$ and $b(x):=a'_1(x)$. Since $a_j^{\ve}$ are even functions, the functions $v_{\ve}(x):=-u_{\ve}(-x)$ also satisfy \er{2.20}. By uniqueness, we must have $v_{\ve}\eq u_{\ve}$, i.e. $u_{\ve}(x)$ are odd functions. In particular, $u_{\ve}(0)=0$, so that $u_{\ve}$ satisfy \er{2.16}. Now \er{2.21} follows immediately from the previous theorem.
\end{proof}

\begin{remark}\label{R2.7}
    The property \er{2.21} implies that for solutions $u_{\ve},\,\ve>0$, to the equation \er{2.20} of mixed type, there are no uniform estimates for the modulus of continuity at the point $x=0$. Moreover, the non-negative functions $\tl{u}_{\ve}:=1+{u}_{\ve}$ also satisfy
    $L^{\ve}\tl{u}_{\ve}=0$ in $(-1,1)$, but the Harnack inequality \er{1.2} fails for $u=\tl{u}_{\ve},\,Y=Y_r=0,\,r=1/2$, because it easily brings to a contradiction:
    \[ 1=\tl{u}_{\ve}(0)
    \le\sup_{(-1/2,\,1/2)} \tl{u}_{\ve}
    \le N\cdot \inf_{(-1/2,\,1/2)} \tl{u}_{\ve}
    =N\cdot\tl{u}_{\ve}\Big(-\frac{1}{2}\Big)
    \to 0\itext{as}\ve\to 0^+.\]
\end{remark}

\begin{remark}\label{R2.8}
    There is a lot of flexibility in the choice of functions $a_j(x)$ satisfying \er{2.19}. We can assume that $|a_j-1|\le\delta_0$ with an arbitrary small $\delta_0>0$. For example, one can take
    \[ a_1:=1+\eta_1,\quad a_2:=1-\eta_1+\eta_2,
    \itext{where} 0\le \eta_1, \eta_2\le \delta_0,\]
    both functions $\eta_1$ and $\eta_2$ are even, belong to $\bS$, vanish near integers $x\in \bZ$, not identically zero, and $\eta_1$ has compact support in the set $\{x:\; \eta'_2(x)>0\}$. Then $\eta_1\eta'_2\ge 0$, and \er{2.19} holds true:
    \[\il_0^1\frac{a'_1(x)}{a_1(x)+a_2(x)}\,dx
    =\il_0^1 \frac{d\eta_1}{2+\eta_2}
    =\il_0^1 \frac{\eta_1\eta'_2}{(2+\eta_2)^2}\,dx
    >0.\]
\end{remark}

\begin{remark}\label{R2.9}
    The assumption \er{2.19} in Corollary \ref{C2.6} is sharp in the following sense: if we replace ``$>$'' by ``$<$'', then instead of \er{2.21}, we will have convergence $u_{\ve}\to 0$ uniformly on every interval $[-\delta,\delta],\;0<\delta<1$. For the proof of this fact, one can note that if in the formulation of Theorem \ref{T2.5}, the assumption $F(1)>0$ is replaced by $F(1)<0$, then $u_{\ve}\to 0$ as $\ve\to 0^+$ uniformly on every interval $[0,\delta],\;0<\delta<1$. In turn, the proof of this statement only requires substitution of the function $h_K$ in \er{2.18}, which satisfies $h''_K+Kh'_K=0$, by
    \[ h_K(x):=\frac{e^{Kx}-1}{e^K-1}
    \itext{satisfying}h''_K-Kh'_K=0.\]
\end{remark}

\section{Existence of periodic solutions to parabolic equations}\label{S3}

In this section, we discuss the existence of solutions to second order parabolic equations
\begin{equation}\label{3.1}
    Lu:=-p\pa_tu+a_{ij}D_{ij}u+b_iD_iu=f,
\end{equation}all functions $p,a_{ij},b_j,f$, and $u$ are {\em periodic both in $t$ and $x$}, $a_{ij}$ satisfy the uniform parabolicity condition \er{U} with a constant $\nu\in (0,1]$, and also $p=p(t,x)\in [\nu,\nu^{-1}]$.
 This requires certain conditions on these functions, which are in different situations are usually referred to as the Fredholm alternative. For elliptic equations, including higher order case, these conditions are presented from the point of view of Fourier analysis in Part 2, Ch. 3 of the book \cite{BJS}. One can adjust the general approach there to our equations \er{3.1}. However, since we need it only in a very particular case, we give a proof for completeness, and also as an application of the Harnack inequality. Formally, by dividing both parts of \er{3.1} by $p$, one can reduce this equation to the case $p=1$, with $\nu^2$ in place of $\nu$.
 However, this is not the case with the adjoint equation $L^*v=0$ in \er{3.2} below.

 \smallskip

Let $l$ and $l_0$ be given positive constants. Denote
    \[ K_0:=[0,l]^n\su \bR^n,\qquad
    K:=[0,l_0]\times K_0\su \bR^{n+1}.\]
We say that the function $u=u(t,x)=u(t,x_1,\ldots,x_n)$ is $K$-periodic if it is defined on the whole space $\bR^{n+1}$, $l_0$-periodic with respect to $t$, and $l$-periodic with respect to each of variables $x_1,\ldots,x_n$. Let $\bS(K)$ denote the class of all $K$-periodic functions $u(t,x)$ on $\bR^{n+1}$, which are continuous in $\bR^{n+1}$ together with all derivatives $\pa_tu, D_iu, D_{ij}u$.
\smallskip

Roughly speaking, the Fredholm alternative states that the equation \er{3.1} has a solution in $\bS(K)$ if and only if $f$ is orthogonal in $\bL^2(K)$ to solutions $v$ of the corresponding homogeneous adjoint equation
\begin{equation}\label{3.2}
    L^*v:=\pa_t(pv)+D_{ij}(a_{ij}v)-D_i(b_iv)=0.
\end{equation}
For simplicity, we assume that all the functions $p,a_{ij},b_i,u,v\,$ belong to $\bS(K)$, so that we deal with classical solutions of equations \er{3.1} and \er{3.2}.

\begin{theorem}\label{T3.1}
    Let $a_{ij}$ satisfy the uniform parabolicity condition \er{U}, $p,\,a_{ij}$, and $b_i$ belong to $\bS(K)$, and let $v\in \bS(K)$ be a strictly positive solution of \er{3.2}. Then for a given function $f\in \bS(K)$, the equation \er{3.1} has a solution $u\in \bS(K)$ if and only if
\begin{equation}\label{3.3}
    (f,v):=\iil_K (fv)(t,x)\,dx\,dt=0.
\end{equation}
\end{theorem}

\begin{proof}
    Let $u\in \bS(K)$ satisfy the equation \er{3.1}. Consider the function
    \begin{equation}\label{3.4}
        I(t):=\il_{K_0}(puv)(t,x)\,dx.
    \end{equation}
    By the smoothness assumptions and equalities \er{3.1}--\er{3.2}, we obtain
    \begin{gather*}
    I'(t)=\il_{K_0}\pa_t(puv)(t,x)\,dx
    =\il_{K_0}\big(vp\,\pa_tu+u\,\pa_t(pv)\big)\,dx  \\
    =\il_{K_0}\Big[v\,\big( a_{ij}D_{ij}u+b_iD_iu-f\big)
    +u\,\big(-D_{ij}(a_{ij}v)+D_i(b_iv)\big)\Big]\,dx.
    \end{gather*}
    Since all the functions $p,a_{ij},b_i,u,v$ are $l$-periodic with respect to $x_i,x_j$, after applying Fubini's theorem and integrating by parts, the terms with $a_{ij}$ and $b_i$ cancel and we get
    \begin{equation}\label{3.5}
    I'(t)=-\il_{K_0}(fv)(t,x)\,dx.
    \end{equation}
    The function $I(t)$ is $l_0$-periodic, because $u$ and $v$ are $l_0$-periodic with respect to $t$. This implies for an arbitrary $T$
    \begin{equation}\label{3.6}
    \il_T^{T+l_0}\bigg[\il_{K_0}(fv)(t,x)\,dx\bigg]\,dt
    =-\il_T^{T+l_0}I'(t)\,dt
    =I(T)-I(T+l_0)=0.
    \end{equation}
    In particular, taking $T=0$, we get the desired equality \er{3.3}, and the theorem is proved in one direction.
    \smallskip

    For the proof in the opposite direction, fix a function $f\in \bS(K)$ satisfying \er{3.3}, and let $u_0=u_0(t,x)$ be a bounded solution to the Cauchy problem
    \[ Lu_0(t,x)=f(t,x)\itext{in}\bR^{n+1}_+:=\{t>0,\;x\in\bR^n\};\qquad
    u_0(0,x)\eq 0.\]
    It is well known (see, e.g. \cite{Fr, LSU}) that there exists a classical solution to this problem, which is unique in the class of bounded functions. From the uniqueness and $\,l$-periodicity of $p,a_{ij},b_i,f$ with respect to $x_k,\; (k=1,\ldots,n)$, it follows that $u_0(t,x+le_k)\eq u_0(t,x)$, i.e. the solution $u_0$ is also $l$-periodic with respect to $x_k$. Note that we did not use periodicity with respect to $t$ in the proof of the equality \er{3.5} for the function $I(t)$ in \er{3.4}. Hence the same equality holds true with $u_0$ instead of $u$, namely,
    \[ I_0(t):=\il_{K_0}(pu_0v)(t,x)\,dx\itext{satisfies}
    I'_0(t)=-\il_{K_0}(fv)(t,x)\,dx.\]
    In addition, we have $I(0)=0$ and the equality \er{3.3}. Therefore,
    \[ I_0(l_0)=\il_0^{l_0}I'_0(t)\,dt
    =-\il_0^{l_0}\bigg[\il_{K_0}(fv)(t,x)\,dx\bigg]\,dt
    =-\iil_K (fv)(t,x)\,dx\,dt
    =0.\]
    Moreover, since $f$ and $v$ are $l_0$-periodic with respect to $t$, the function $I'_0(t)$ is also $l_0$-periodic, and we obtain
    \begin{equation}\label{3.7}
     I_0(kl_0):=\il_{K_0}(pu_0v)(kl_0,x)\,dx=0
     \itext{for}k=0,1,2,\ldots.
    \end{equation}

    Further, consider the functions
    \begin{equation}\label{3.8}
    u_k(t,x):=u_0(t+kl_0,x),\quad w_k(t,x):=u_{k+1}(t,x)-u_k(t,x)
    \itext{on}\bR^{n+1}_+
    \end{equation}
    for $\,k=0,1,2,\ldots$. Since all the given functions are $l_0$-periodic with respect to $t$, we have
    \[ Lw_k=0\itext{in}\bR^{n+1}_+
    \itext{for}k=0,1,2,\ldots.
    \]
    We want to show that the functions $u_k$ converge uniformly on $\bR^{n+1}_+$ to a solution $u\in \bS(K)$ of the equation \er{3.1}. For this purpose, we will use the estimate for oscillation similar to \er{1.3} applied to the functions $w_k$. Denote
    \[c_k:=\osc_{\bR^{n+1}_+}w_k
    :=\sup_{\bR^{n+1}_+}w_k-\inf_{\bR^{n+1}_+}w_k
    \itext{for}k=0,1,2,\ldots.\]
    By the maximum principle and $l$-periodicity with respect to $x_1,\ldots,x_n$, we have
    \[ c_k=\osc_{\bR^n}w_k(0,x)=\osc_{K_0}w_k(0,x)
    \itext{for}k=0,1,2,\ldots.\]
    Fix a ball $B_r:=B_r(0)\su\bR^n$ such that $K_0:=[0,l]^n\su B_r$. Without loss of generality, we can assume that $l_0=4r^2$ (using rescaling $t\to\con \cdot t\,$ if necessary). Then for $Y:=(l_0,0)=(4r^2,0)$, we have $C_r(Y)\su C_{2r}(Y)\su \bR^{n+1}_+$.
    \smallskip

    Note that in Theorem \ref{T1.1}, the Harnack inequality is formulated in a simplified form, simultaneously for divergence
    without lower order terms $b_i$, with a constant $N$ depending only on $n$ and $\nu$. In fact, it is true for more general equations \er{3.1} with $N$ depending also on $r$ and on the constant $M\ge |b_i|$ (see \cite{KS, K, Li}). Therefore, the estimate \er{1.3},
    \[ \osc_{C_r(Y)}w_k \le \theta \cdot \osc_{C_{2r}(Y)}w_k
    \]
    is also true with a constant $\theta \in (0,1)\,$ which does not depend on $k$. Further, from \er{3.8} it follows
    \[w_{k+1}(t,x)\eq w_k(t+l_0,x)\itext{on} \bR^{n+1}_+.
    \]
    \smallskip

    Combining together the previous relations, and also $\{l_0\}\times K_0\su\ol{C_r(Y)}$, we find:
    \[c_{k+1}=\osc_{K_0}\,w_{k+1}(0,x)
     =\osc_{K_0}\,w_k(l_0,x)
     \le
     \osc_{C_r(Y)}w_k \le \theta \cdot \osc_{C_{2r}(Y)}w_k
     \le \theta \cdot c_k. \]
     By induction,
    \[ c_k\le \theta^k\,c_0
    \itext{for} k=0,1,2,\ldots.\]
    Moreover, by virtue of \er{3.7},
    \[ \il_{K_0}(pw_kv)(0,x)\,dx
    =\il_{K_0}(pu_0v)(kl_0,x)\,dx - \il_{K_0}(pu_0v)\big((k+1)l_0,x\big)\,dx =0.\]
    Since $p,v>0$, the functions $w_k$ must change sign on $\bR^{n+1}_+$ unless $u_k(0,x)\eq 0$. This guarantees the estimate
    \begin{equation}\label{3.9}
    \sup_{\bR^{n+1}_+}\,|w_k|
    \le 2\,\osc_{\bR^{n+1}_+}\,w_k
    =:2c_k\le 2\,\theta^k\,c_0
    \itext{for} k=0,1,2,\ldots.
    \end{equation}
    For $k\ge1$,  we have
    \[ u_k=u_0+w_0+w_1+\cdots+w_{k-1}.\]
    The estimate \er{3.9} implies that $u_k$ is a Cauchy sequence, so it converges uniformly on $\bR^{n+1}_+$. In addition, $Lu_k=f$ in $\bR^{n+1}_+$ for all $k$. By standard results in the theory of parabolic equations with smooth coefficients, the limit function $u=u(t,x)$ is smooth and satisfies $Lu=f$ in $\bR^{n+1}_+$. Finally,
    \[ u(l_0,x)-u(0,x)=\lim_{k\to\8}\big[u_k(l_0,x)-u_k(0,x)\big]
    =\lim_{k\to\8}w_k(0,x)=0.\]
    This means that $u(t,x)$ is $l_0$-periodic with respect to $t$, and therefore $u\in \bS(K)$. The theorem is proved.
\end{proof}

\section{Parabolic equations of mixed type}\label{S4}

We consider second order parabolic equation \er{1.7} in the case $x\in \bR^1$:
\begin{equation}\label{4.1}
    Lu:=-pu_t+(au_x)_x=0.
\end{equation}
Here $p=p(t,x),\,a=a(t,x)\in [\nu,\nu^{-1}]$ for some constant $\nu\in (0,1]$, and the indices $t$ and $x$ indicate differentiation with respect to the corresponding variable:
$u_t:=\pa_t u,\,u_x:=\pa_x u$, etc. Our goal is to show that in this case, the Harnack inequality (Theorem \ref{T1.1}) and the H\"{o}lder estimate (Corollary \ref{C1.3}) fail in general. For simplicity, we assume that all the functions in this section are smooth, so that all the derivatives are understood in the classical sense. We can rewrite \er{4.1} in the form
\begin{equation}\label{4.2}
    Lu=-pu_t+au_{xx}+bu_x=0,
    \itext{where}b:=a_x.
\end{equation}
Without additional smoothness assumptions, we do not have control of the coefficient $b$. Nevertheless, if we assume that all the given data are periodic in both $t$ and $x$, then we still can use the Fredholm alternative result in Theorem \ref{T3.1}: the equation $Lu=f$ has a periodic solution if and only if $(f,v)=0$, where $v=v(t,x)$ is a strictly positive solution of the adjoint equation
\begin{equation}\label{4.3}
    L^*v:=(pv)_t+(av_x)_x=0.
\end{equation}

This result is reduced to more elementary Theorem \ref{T2.2} if we assume that $p,a,u,v$ are functions of one variable $y:=t+x$.  In order to comply with Section \ref{S2}, we keep same notations for these functions as in \er{4.1}--\er{4.3}, where they are functions of two variables $t$ and $x$. Then the above equalities are simplified as follows:
\begin{equation}\label{4.4}
    Lu:=-pu'+(au')'=au''+(b-p)u',
    \itext{where}b:=a',
\end{equation}
and
\begin{equation}\label{4.5}
    L^*v:=(pv)'+(av')'=0.
\end{equation}
We assume that all the functions in \er{4.4}--\er{4.5} belong to the class $\bS$ of all smooth $1$-periodic functions on $\bR^1$, and as in Theorem \ref{T2.2}, $v$ is a strictly positive solution of \er{4.5} in $\bS$ satisfying the condition
\begin{equation}\label{4.6}
    \il_0^1v(x)\,dx=1.
\end{equation}

For our purposes, we need to have $b_0:=(b,v)\ne 0$. The operator $L$ in \er{4.4} is different from $Lu:=au''+bu'$ in \er{2.3} by the presence of an additional term $p$. Because of this term, we do not have such a simple relation between $b_0$ and the coefficients of $L$ as in Corollary \ref{C2.4}, so we just use a very special particular case. Similarly to Remark \ref{R2.8}, take not identically zero functions $\eta_1, \eta_2\in \bS$, such that
    \[ 0\le \eta_1\le\delta_0, \quad
    |\eta_2|\le \delta_0,\quad
    \il_0^1 \eta_2(x)\,dx=0,\]
and $\eta_1$ has compact support in the set $\{x:\; \eta'_2(x)>0\}$.
Then the functions $a:=1+\eta_1,\, v:=1+\eta_2\in\bS\,$ satisfy \er{4.6} and
\begin{equation}\label{4.7}
    b^0:=(b,v):=\il_0^1 bv\,dx=\il_0^1 a'v\,dx
    =\il_0^1 \eta'_1(1+\eta_2)\,dx
    =\il_0^1 \eta'_1\eta_2\,dx>0.
\end{equation}
Obviously, we can make $v$ close to $1$, $a'v$ close to $\,0$, and choose $\,p\in\bS\,$ close to $1$ from the identity
\begin{equation}\label{4.8}
    av'+pv\eq 1,
\end{equation}
which in turn implies \er{4.5}.
\smallskip

We remind that by Corollary \ref{C2.3}, for an arbitrary $f\in\bS$ and $f^0:=(f,v)$, the equation $Lu=f-f^0$ is solvable in $\bS$.

\begin{theorem}\label{T4.1}
    For $p,a,b:=a'\in\bS\,$ and $\ve>0$, denote
    \begin{equation}\label{4.9}
    p^{\ve}:=p(y_{\ve}),\quad
    a^{\ve}:=a(y_{\ve}),\quad
    b^{\ve}:=b(y_{\ve}),
    \itext{where} y_{\ve}:=\ve^{-2}t+\ve^{-1}x,
    \end{equation}
    and also $p^0:=(p,v)>0,\,b^0:=(b,v)$, where $v\in\bS$ satisfies \er{4.5}--\er{4.6}. Let $g$ be a given function in $C_0^{\8}(\bR^1)$, and let $u^{\ve}=u^{\ve}(t,x)$ be a bounded solution to the Cauchy problem
    \begin{equation}\label{4.10}
    L^{\ve}u^{\ve}:=
    -p^{\ve}u_t^{\ve}+(a^{\ve}u_x^{\ve})_x=0,\quad t>0;
    \qquad u^{\ve}(0,x)=g(x).
    \end{equation}
    Then
    \begin{equation}\label{4.11}
    |u^{\ve}(t,x)-U^{\ve}(t,x)|\le N\cdot(\ve+t),\quad t>0;
    \end{equation}
    where
    \begin{equation}\label{4.12}
    U^{\ve}(t,x):=g(c\ve^{-1}t+x),\qquad
    c:=(p^0)^{-1}b^0,
    \end{equation}
    and the constant $N>0$ does not depend on $\,\ve>0$.
\end{theorem}

\begin{proof}
    Note that $U^{\ve}$ is a solution to the {\em transport equation}
    \begin{equation}\label{4.13}
    L^{0\ve}U^{\ve}:=-p^0U_t^{\ve}+\ve^{-1}b^0U_x^{\ve}=0,\quad t>0;
    \qquad U^{\ve}(0,x)=g(x).
    \end{equation}
    By Corollary \ref{C2.3}, there exist functions $P$ and $B$ in $\bS$ satisfying
    \[ LP=p-p^0,\qquad LB=b-b^0.\]
    Consider the functions
    \begin{equation}\label{4.14}
     w^{\ve}:=u^{\ve}-U^{\ve}
    + \ve^2\,\big(P^{\ve}U_t^{\ve}-\ve^{-1}B^{\ve}U_x^{\ve}\big),
    \quad\ve>0.
    \end{equation}
    where $\,P^{\ve}:=P(y_{\ve}),\; B^{\ve}:=B(y_{\ve})$.
    Similarly to the proof of Theorem \ref{T2.5},
    \[ L^{\ve}w^{\ve}=-L^{\ve}U^{\ve}+I_1+I_2+I_3,\]
    where
    \begin{eqnarray*}
      I_1 &:= & (p^{\ve}-p^0)U_t^{\ve}
      -\ve^{-1}(b^{\ve}-b^0)U_x^{\ve},\\
      I_2 &:= & \ve^{2}\big(P^{\ve}L^{\ve}U_t^{\ve}
      -\ve^{-1}B^{\ve}L^{\ve}U_x^{\ve}\big),\\
      I_3 &:= & 2\ve^2a^{\ve}\big(P_x^{\ve}U_{tx}^{\ve}
      -\ve^{-1}B_x^{\ve}U_{xx}^{\ve}\big).
    \end{eqnarray*}

    Since
    \[ L^{\ve}U^{\ve}
    :=p^{\ve}U_t^{\ve}-(a^{\ve}U_x^{\ve})_x
    =p^{\ve}U_t^{\ve}-a^{\ve}U_{xx}^{\ve}
    -\ve^{-1}b^{\ve}U_x^{\ve},\]
    and $U^{\ve}$ satisfies \er{4.13}, we have
    \[ -L^{\ve}U^{\ve}+I_1
    =-a^{\ve}U_{xx}^{\ve}
    =-a^{\ve}g''(c\,\ve^{-1}t+x),\]
    which is uniformly bounded  with respect to $\ve>0$.
    Moreover, the derivatives $P_x^{\ve}=\ve^{-1}P'(y_{\ve})$ and
    $B_x^{\ve}=\ve^{-1}B'(y_{\ve})$ are of order $\ve^{-1}$. From the explicit expressions for $L^{\ve}$ in \er{4.10} and $U^{\ve}$ in \er{4.12}, it follows that all the terms in $I_2$ and $I_3$ are also uniformly bounded. Hence we have
    \begin{equation}\label{4.15}
        |L^{\ve}w^{\ve}|\le N_1
        \itext{for all}\ve>0,
    \end{equation}
    and by virtue of \er{4.14},
    \begin{equation}\label{4.16}
        |u^{\ve}-U^{\ve}-w^{\ve}|\le N_2\ve
        \itext{for all}\ve>0.
    \end{equation}
    The constants $N_1$ and $N_2$ in these estimates do not depend on $\ve>0$. In particular, $|w^{\ve}(0,x)|\le N_2\ve$.
    \smallskip

    Now we can compare the functions $\pm w^{\ve}$ with $W^{\ve}(t):=N_1\nu^{-1}t+N_2\ve$. Since $\,p\ge \nu>0$, we have
    \[ L(\pm w^{\ve})\ge -N_1\ge -N_1\nu^{-1}p=LW^{\ve},
    \itext{and}
    \pm w^{\ve}(0,x)\le N_2\ve =W^{\ve}(0).\]
    By the maximum principle,
    $\pm w^{\ve}(t,x)\le W^{\ve}(t)$ for all $t\ge 0,\,x\in\bR^1$. Using \er{4.16} once again, we get the desired estimate \er{4.11}:
    \[ |u^{\ve}-U^{\ve}|
    \le |w^{\ve}|+N_2\ve
    \le N_1\nu^{-1}t+2N_2\ve
    \le N\cdot (t+\ve)\]
    with $\,N:=\max\{N_1\nu^{-1}, 2N_2\}$.
\end{proof}

\begin{theorem}\label{T4.2}
    Under assumptions of the previous theorem, suppose that $b^0:=(b,v)\ne 0$. Then the following statements hold true.
    \smallskip

    (I) For arbitrary $Y:=(s,y)\in\bR^2,\,r>0$, and $\delta \in(0,1)$, there are solutions $u^{\ve}$ to the equation $L^{\ve}u^{\ve}=0$ in $C_r(Y)$ satisfying
    \begin{equation}\label{4.17}
    \osc_{C_{\delta r}(Y)}u^{\ve}\ge
    (1-\delta)\cdot \osc_{C_r(Y)}u^{\ve}.
    \end{equation}
    In particular, there is no H\"{o}lder estimate \er{1.4} for $u=u^{\ve}$ with a constant $\alpha>0$ independent on $\ve>0$.
    \smallskip

    (II) For arbitrary $Y:=(s,y)\in\bR^2,\,r>0$, and $\delta \in(0,1)$, there are solutions $u^{\ve}$ to the equation
    $L^{\ve}u^{\ve}=0$ in $C_{2r}(Y)$ satisfying
    $0\le u^{\ve}\le 1$ in $C_{2r}(Y)$,
    \begin{equation}\label{4.18}
    \sup_{C_r(Y)}u^{\ve}\le\delta,\itext{and}
    \sup_{C_r(Y_r)}u^{\ve}\ge 1-\delta,\itext{where}
    Y_r:=(s-2r^2,y).
    \end{equation}
    In particular, there is no Harnack inequality \er{1.2} with a constant $N>1$ independent on $\ve>0$.
\end{theorem}

\begin{proof}
    (I) By rescaling $(t,x)\to (r^{-2}t,r^{-1}x)$, the proof is reduced to the case $r=1$. In addition, replacing $x$ by $-x$ if necessary, we can assume $b^0<0$. If we prove our statement for $Y=Y^{\ve}:=(t^{\ve},2)$ with a convenient choice of $t^{\ve}>0$, then the general case of $Y\in\bR^2$ is covered by parallel translation in $\bR^2$.
    \smallskip

    Fix $\delta \in(0,1)$ and take an arbitrary function $\,g\,$ such that
    \begin{equation}\label{4.19}
    g\in C_0^{\8}(\bR^1),\quad
    0\le g\le 1,\quad
    g(0)=1,
    \itext{and} g(x)\eq 0
    \itext{for} |x|\ge \delta,
    \end{equation}
    and let $u^{\ve}$ be a bounded solution to the problem $\er{4.10}$. The constant $N$ in \er{4.11} depends on the original data $p,a$, and $g$ (which in turn depends on $\delta$), but not on $\ve>0$. The function $U^{\ve}$ in \er{4.12} satisfies

    \begin{equation}\label{4.20}
    U^{\ve}(t^{\ve},x)\eq g(x-2),\itext{where}
    t^{\ve}:=-\frac{2\ve}{c}
    =-\frac{2p^0\ve}{b^0}>0.
    \end{equation}
    From \er{4.11} it follows
    \begin{equation}\label{4.21}
    |u^{\ve}(t^{\ve},x)-g(x-2)|
    =|u^{\ve}(t^{\ve},x)-U^{\ve}(t^{\ve},x)|
    \le N\cdot (\ve+t^{\ve})\le\frac{\delta}{2},
    \end{equation}
    provided $\ve>0$ is small enough. In particular,
    \begin{equation}\label{4.22}
    |u^{\ve}(t^{\ve},2)-1|\le\frac{\delta}{2},\qquad
     |u^{\ve}(t^{\ve},2\pm \delta) |\le\frac{\delta}{2}.
    \end{equation}

    By the maximum principle, $0\le u^{\ve}(t,x)\le 1$ for all $t>0,\,x\in\bR^1$. Moreover, by standard extension $u^{\ve}(t,x)\eq 0$ on $C_1(Y^{\ve})\cap \{t\le 0\}$, we get a classical solution of $L^{\ve}u^{\ve}=0$ in $C_1(Y^{\ve}),\;Y^{\ve}:=(t^{\ve},2)$. Indeed, if $\tl{u}_{\ve}$ is a classical solution of $L^{\ve}\tl{u}_{\ve}=0$ in $C_1(Y^{\ve})$ with the given data on the parabolic boundary of $C_1(Y^{\ve})$:
    \begin{eqnarray*}
    \tl{u}_{\ve}\eq u^{\ve} &&
    \text{on}\quad (0,t^{\ve})\times \{1,3\},\\
    \tl{u}_{\ve}\eq 0 &&
    \text{on}\quad
    \big((t^{\ve}-1,0]\times\{1,3\}\big)\cup
    \big(\{t^{\ve}-1\}\times [1,3]\big),
    \end{eqnarray*}
    then by uniqueness, we must have
    \[ \tl{u}_{\ve}\eq 0\itext{on}
    C_1(Y^{\ve})\cap \{t\le 0\}, \itext{and}
    \tl{u}_{\ve}\eq u^{\ve}\itext{on}
    C_1(Y^{\ve})\cap \{t>0\}.\]

    Finally, from \er{4.22} it follows
    \begin{equation}\label{4.23}
    \sup_{C_{\delta}(Y^{\ve})} u^{\ve}
    \ge u^{\ve}(Y^{\ve})\ge 1-\frac{\delta}{2},\qquad
    \inf_{C_{\delta}(Y^{\ve})} u^{\ve}
    \le u^{\ve}(t^{\ve},2\pm\delta)\le \frac{\delta}{2},
    \end{equation}
    and we get the property \er{4.17}:
    \[\osc_{C_{\delta}(Y^{\ve})} u^{\ve}
    \ge 1-\delta
    \ge (1-\delta)\cdot
    \osc_{C_1(Y^{\ve})} u^{\ve},\]

    (II) Fix a constant $\delta \in (0,1)$. It suffices to find a particular combination of $Y:=(s,y),\,r>0$, and $u^{\ve}$ with $\ve>0$, which satisfies \er{4.18}; then the case of general $Y\in \bR^2$ and $r>0$ follows by translation and rescaling argument as in the proof of the previous part (I).
    \smallskip

    Using the previous construction, fix a function $g$ in \er{4.19} and $u^{\ve}$ satisfying \er{4.10}, where $\ve>0$ is small enough to guarantee the estimate \er{4.21}. Then choose $r>0$ from the equality $t^{\ve}=2r^2$, and set $s:=2t^{\ve}=4r^2,\,Y:=(s,2)$. Obviously, we can assume that both $\delta$ and $r$ belong to $(0,1/2)$. By the choice of $t^{\ve}$ in \er{4.20}, we have $c\ve^{-1}r^2=-1$. Hence
    \[ c\ve^{-1}t+x<-\frac{1}{2}\itext{for every}
    X:=(t,x)\in C_r(Y):=(3r^2,4r^2)\times (2-r,2+r),\]
    so that $U^{\ve}=0$ on $C_r(Y)$, and by \er{4.11}, \er{4.21},
    \[ u^{\ve}(t,x)=u^{\ve}(t,x)-U^{\ve}(t,x)
    \le N\cdot (\ve+t)\le N\cdot (\ve+2t^{\ve})
    \le\delta \itext{on} C_r(Y),\]
    i.e. we get the first estimate in \er{4.18}. The second estimate is contained in \er{4.23}, because $Y_r:=(s-2r^2,2)=(t^{\ve},2)=:Y^{\ve}$:
    \[ \sup_{C_r(Y_r)}u^{\ve}
    \ge u^{\ve}(Y_r)=u^{\ve}(Y^{\ve})
    \ge 1-\frac{\delta}{2}>1-\delta .\]
    Theorem is proved.
\end{proof}


\begin{thebibliography}{}\itemsep5pt

\bibitem[BJS]{BJS} L. Bers, F. John, and M. Schecter, {\em Partial differential equatins}, Interscience Publishers, New York--London--Sydney, 1964.

\bibitem[BLP]{BLP} A. Bensoussan, J.~L. Lions, and G. Papanicolaou, {\em
Asymptotic Analysis for Periodic Structures}, North-Holland Publishing Co.,
Amsterdam--New York--Oxford, 1978.

\bibitem[DG]{DG} E. De Giorgi, {\em Sulla differenziabilit\`{a} e
    l`analiticit\`{a} delle estremali degli integrali multipli regolari}, Mem. Accad. Sci. Torino Cl. Sci. Fis. Mat. Natur. (3) \textbf{3} (1957), 25--43.

\bibitem[F]{F} M.Freidlin, {\em Dirichlet's problem for an equation with periodic coefficients depending on a small parameter}, Theory of probability and its application, Vol 9, (1964), 121--125.

\bibitem[Fr]{Fr}  A. Friedman, {\em Partial Differential Equations of Parabolic Type}, Prentice Hall, 1964.

\bibitem[FS]{FS} E. Ferretti and M. V. Safonov, {\em Growth
    theorems and Harnack inequality for second order parabolic equations}, Harmonic analysis and boundary value problems (Fayetteville, ARA, 2000), 87--112, Contemp. Math., \textbf{277}, Amer. Math. Soc., Providence, RI, 2001.

\bibitem[GT]{GT} D. Gilbarg and N. S. Trudinger, {\em Elliptic Partial Differential Equations of Second Order}, Springer-Verlag, 1983.

\bibitem[JKO]{JKO} V. V. Jikov, S. M. Kozlov, and O. A. Oleinik, {\em Homogenization of Differential Operators and Integral functionals}, Springer-Verlag, 1994.

\bibitem[H]{H} P. Hess, {\em Periodic-Parabolic Boundarg Value Problems and Positivity}, Pitman, 1991.

\bibitem[K]{K} N. V. Krylov, {\em Nonlinear Elliptic and Parabolic
Equations of Second Order}, Nauka, Moscow, 1985 in Russian; English
translation: Reidel, Dordrecht, 1987.

\bibitem[KS]{KS} N. V. Krylov and M. V. Safonov,  {\em A certain
    property of solutions of parabolic equations with measurable coefficients}, Izvestia Akad. Nauk SSSR, ser. Matem. \textbf{44}, no. 1 (1980), 161--175 in Russian; English translation in Math. USSR Izvestija, \textbf{16}, no. 1 (1981), 151--164.

\bibitem[La]{La}  E. M. Landis, {\em Second Order Equations of Elliptic and Parabolic Type,} ``Nauka'', Moscow, 1971 in Russian; English transl.: Amer. Math. Soc., Providence, RI, 1997.

\bibitem[Li]{Li}  G. M. Lieberman, {\em Second Order Parabolic Differential Equations}, World Scientific, 1996.

\bibitem[Li1]{Li1}  G. M. Lieberman, {\em Time-periodic solutions
of linear parabolic differential equations}, Comm. Partial Diff. Equations \textbf{24} (1999), 631--663.

\bibitem[LSU]{LSU}  O. A. Ladyzhenskaya, V. A. Solonnikov, and N. N. Ural'tseva, {\em Linear and Quasi-linear Equations of Parabolic Type}, Nauka, Moscow, 1967 in Russian; English transl.: Amer. Math. Soc., Providence, RI, 1968.

\bibitem[M1]{M1} J. Moser, {\em On Harnack's theorem for elliptic
    differential equation}, Comm. Pure Appl. Math. \textbf{14} (1961), 577--591.

\bibitem[M2]{M2}  J. Moser, {\em A Harnack inequality for parabolic
    differential equations}, Comm. Pure and Appl. Math. \textbf{17} (1964), 101--134; and correction in: Comm. Pure and Appl. Math. \textbf{20} (1967), 231--236.

\bibitem[N]{N} J. Nash, {\em Continuity of solutions of parabolic and elliptic equations}, Amer. J. Math. \textbf{80} (1958), 931--954.

\bibitem[Nd]{Nd} N. Nadirashvili, {\em Nonuniqueness in the martingale problem and the Dirichlet problem for uniformly elliptic operators}, Ann. Scuola Norm. Sup. Pisa Cl. Sci. (4) \textbf{24}, no. 3 (1997), 537-–549.

\bibitem[NU]{NU} A. I. Nazarov and N. N. Uraltseva, {\em The Harnack inequality and related properties for solutions of elliptic and parabolic equations with divergence-free lower-order coefficients},  Algebra i Analiz \textbf{28}, no. 1 (2011) in Russian] ; English transl.: St. Petersburg Math. J. \textbf{23}, no. 1 (2012), 93--115 .

\bibitem[PE]{PE}  F. O. Porper and S. D. Eidelman, {\em Two-sided estimates of fundamental solutions of second-order parabolic equations, and some applications}, Uspekhi Mat. Nauk \textbf{39}, no. 3 (1984), 107--156 in Russian; English transl. in Russian Math. Surveys \textbf{39}, no. 3 (1984), 119--178.

\bibitem[S]{S} M. V. Safonov, {\em Nonuniqueness for second order elliptic equations with measurable coefficients}, SIAM J. Math. Anal. {\bf 30} (1999), 879--895.

\end{thebibliography}
\end{document}